\documentclass[11pt]{article}
\usepackage{geometry}                
\usepackage[parfill]{parskip}    
\usepackage{graphicx}
\usepackage{amssymb,amsmath,amsthm,amsfonts}
\usepackage{lipsum}
\newtheorem{theorem}{Theorem}
\newtheorem{lemma}{Lemma}
\usepackage{epstopdf}
\usepackage{enumerate}
\usepackage{fullpage}  
\newcommand{\bbN}{\mathbb{N}}

\renewcommand{\epsilon}{\varepsilon}

\makeatletter
\def\thm@space@setup{%
  \thm@preskip=1cm plus 1cm minus 1cm
  \thm@postskip=\thm@preskip
}
\makeatother

\author{Nicholas Bhattacharya \thanks{University of California, Berkeley.
{\tt nick\_bhat@berkely.edu}}
\and Mark Perlman \thanks{Stanford University. {\tt mpp@stanford.edu}}
}

\title{Time-Inhomogeneous Branching Processes Conditioned on Non-Extinction
}

\date{}                                           

\begin{document}

\maketitle

\begin{abstract}
  In this paper, we consider time-inhomogeneous branching processes and
  time-inhomogeneous birth-and-death processes, in which the offspring distribution and birth and death
  rates (respectively) vary in time. A classical result of branching processes states that
  in the critical regime, a process conditioned on non-extinction and normalized will converge
  in distribution to a standard exponential. In a paper of Jagers \cite{Jagers},
  time-inhomogeneous branching processes are shown to exhibit this convergence as well.
  In this paper, the hypotheses of Jagers' result are relaxed, further hypotheses
  are presented for convergence in moments, and the result is extended to the
  continuous-time analogue of time-inhomogeneous birth-and-death processes.
  In particular, the new hypotheses suggest a simple characterization of the critical regime.
\end{abstract}

\textbf{Keywords:} Branching processes, Birth-and-death processes, Time-inhomogeneous,
Time-independent, Non-extinction

\textbf{Mathematics Subject Classification (2000):} 60J80

\textit{This work was started at the MAPS-REU
led by Prof. D. Dolgopyat and Prof. L. Koralov at the University of Maryland in the summer of 2015.}

\section{Introduction} 
\label{intro}

A branching process is a sequence of random variables $\{Z_{n}\}$ that represents a population of particles
at times steps $n \in \bbN$, so that $Z_{n}$ is the population of the $n$-th generation of particles. 
The initial population is one particle, $Z_{0}=1$, and subsequent generations are defined inductively.
To create the generation $Z_{n+1}$, each of the $Z_{n}$ particles reproduces independently
according to a given offspring distribution, and the sum of the resulting particles quantifies the
next generation $Z_{n+1}$. In classical, time-homogeneous branching processes, the offspring
distribution $X$ is common to all particles. In time-inhomogeneous branching processes, the
offspring distribution $X_{n}$ depends on the time step $n$.
Thus a time-inhomogeneous branching process is defined by
  \begin{align*}
  \begin{cases}
    Z_{0} = 1 \\
    Z_{n+1} = \sum_{j=1}^{Z_{n}}X_{n,j}, \qquad X_{n,j} \text{ are i.i.d. copies of } X_{n}.
  \end{cases}
  \end{align*}
Classical branching processes are categorized into three regimes depending on the expected number of offspring $E(X)$:
a branching process is subcritical if $E(X)<1$, critical if $E(X)=1$ and supercritical if $E(X)>1$. Two important results
pertaining to the asymptotics of classical branching processes can be found in \cite{AN}: the first states
that extinction (the event that $Z_{n}$ is ever 0) occurs with probability
one if and only if $E(X) \leq 1$; the second states that when a branching process in the critical regime is conditioned on
non-extinction and normalized, it converges in distribution to a standard exponential:
  \begin{align*}
    \frac{\zeta_{n}}{E(\zeta_{n})} \to \exp(1),
  \end{align*}
where $\zeta_{n}$ denotes $\{Z_{n} \mid Z_{n} \ne 0\},$ that is, $Z_{n}$ conditioned on non-extinction.
These results hold in the time-inhomogeneous setting, though under more complicated hypotheses. As discussed in \cite{Agresti}, extinction occurs with probability one if and only if
  \begin{align*}
    \sum_{n=0}^{\infty} \frac{1}{E(Z_{n})} = \infty,
  \end{align*}
under some natural regularity conditions on the offspring distribution which are thoroughly discussed in \cite{Jirina}.
It follows from this result that the above equation naturally defines the subcritical regime for time-inhomogeneous
branching processes.

The exponential limit of the conditioned, normalized process has been proven in the time-inhomogeneous setting by
Jagers \cite{Jagers}, under the main hypothesis that
  \begin{align*}
    0 < \lim_{n \to \infty }E(Z_{n}) < \infty,
  \end{align*}
along with regularity conditions. In this paper, a strictly weaker set of hypotheses is shown to be sufficient for the
exponential limit, thus also providing a broader definition of the critical regime.

In the second part of this paper, the two asymptotic results mentioned above (extinction criterion and exponential
limit) are reproved in the setting of time-inhomogeneous
birth-and-death processes. A time-inhomogeneous birth-and-death process $Z_{t}$ is a continuous-time analogue
of a time-inhomogeneous branching process, in which births and deaths occur at random times determined by
the current total population and in terms of infinitesimal rates:
  \begin{align*}
  \begin{cases}
    P(Z_{t+\Delta}-Z_{t}= k \mid Z_{t}) = Z_{t}\Delta b_{k}(t) + o(\Delta), & (k = -1, 1, ..., n) \\
      P(Z_{t+\Delta}-Z_{t}= 0 \mid Z_{t}) = 1-Z_{t}\Delta(b_{-1}(t)+b_{1}(t)+...+b_{n}(t)) + o(\Delta).
  \end{cases}
  \end{align*}
The rates $b_{k}$ are measurable functions defined on $[0,\infty)$.
Birth-and-death processes can be considered a continuous-time analogue to 
branching processes, wherein individual particles reproduce or die at distinct, continuous, random times rather
than reproducing all at once at every integer time.
The birth rates $b_{1}(t),...,b_{n}(t)$ and the death rate $b_{-1}(t)$ replace the offspring distribution
$X_{n}$ of the discrete case. A noteworthy difference between the discrete and continuous
models is that in the continuous model, the population can only increase by a fixed maximum amount
at any reproduction event, whereas the support of the discrete $X_{n}$ could be all of $\bbN$.

In \cite{CM}, extinction is shown to occur with probability one under the condition
  \begin{align*}
    \int_{0}^{\infty}\frac{1}{E(Z_{t})}dt = \infty.
  \end{align*}
In \cite{Kendall}, both asymptotic results are shown in the special ``fractional linear'' case when there is only
the single birth rate $b_{1}$. In this paper we will extend the results to the case of finitely many birth rates.

For the remainder of the paper, the discrete setting is dealt with in its entirety before the continuous setting.

\section{The Discrete Setting} 
\label{The Discrete Setting}

\subsection{Notation} 
\label{Notation}
We denote $g_{n}(s) = E(s^{X_{n}})$ for $s \in [0,1]$, the generating function of the $n$th offspring distribution,
and we denote $f_{n}(s) = E(s^{Z_{n}})$ for $s \in [0,1]$, the generating function of the $n$th generation distribution.
As in the case of classical branching processes, the generating function for the population can be obtained by
composing generating functions of the offspring distributions:
  \begin{align*}
    f_{n}(s) = (g_{0} \circ ... \circ g_{n-1})(s).
  \end{align*}
We'll also make use of composing only some generating functions of the offspring distributions, to obtain the
generating function for $\{Z_{n} \mid Z_{j}=1\}$, a process starting with a single particle at the $j$th generation:
  \begin{align*}
    f_{jn}(s) = (g_{j} \circ g_{j+1} \circ ... \circ g_{n-1})(s).
  \end{align*}
We denote $\phi_{n} = P(Z_{n} \ne 0)$, the probability of non-extinction by the $n$th generation, and analogously
we denote $\phi_{jn} = P(Z_{n} \ne 0 \mid Z_{j}=1)$, the probability of non-extinction by the $n$th generation
if we had started the process with one particle at the $j$th generation. We denote 
the means of the offspring distribution and the population distribution by
  \begin{align*}
    \mu_{n} & = E(X_{n}) = g_{n}'(1) \\
    m_{n} & = E(Z_{n}) = f_{n}'(1).
  \end{align*}
Our expression of $f_{n}$ in terms of the $g_{n}$ tells us that the $n$th population mean is the
product of the preceding generation means:
  \begin{align*}
    m_{n} = \prod_{i=0}^{n-1}\mu_{i}.
  \end{align*}
We will more generally denote the moments and factorial moments of the population distribution by
  \begin{align*}
    M_{n,r} & = E(Z_{n}^{r}) \\
    F_{n,r} & = f_{n}^{(r)}(1).
  \end{align*}
We will denote by $\zeta_{n}$ the process $Z_{n}$ conditioned on non-extinction:
  \begin{align*}
    \zeta_{n} = \{Z_{n} \mid Z_{n} \ne 0\}.
  \end{align*}
Finally, we introduce the quantity
  \begin{align*}
    \Gamma_{n} & = \sum_{j=0}^{n-1}\frac{g_{j}''(1)}{2\mu_{j}m_{j+1}},
  \end{align*}
which we will show is asymptotically equal to $\phi_{n}^{-1}$ in the critical regime.

\subsection{Hypotheses} 
\label{Hypotheses}

We'll make use of the following hypotheses:
\begin{enumerate}
  \item[(H1)] $\sup_{n \geq 0} P(X_{n} = 0) < 1$
  \item[(H2)] $\limsup_{n \to \infty} \left[P(X_{n}=0)+P(X_{n}=1)\right] < 1$
  \item[(H3)] $\liminf_{n \to \infty} P(X_{n}=0)  > 0$
  \item[(H4a)] $\sup_{n \geq 0} E(X_{n}^{r}) < \infty$ for $r=1,2,3$
  \item[(H4b)] $\sup_{n \geq 0} E(X_{n}^{r}) < \infty$ for every $r \geq 1$
  \item[(H5)] $\Gamma_{n} \to \infty$
  \item[(H6)] $m_{n}\Gamma_{n} \to \infty$.
\end{enumerate}
Hypotheses (H1)-(H4) are the regularity conditions on the offspring distribution mentioned before.
Hypothesis (H1) prevents the branching process from entirely dying out in a single generation.
Hypothesis (H2) prevents the possibility of every particle generating almost exactly one in the next generation,
thereby causing the process to stagnate or simply stop. Hypotheses (H1) and (H2) can be replaced by the
slightly stronger assumption
  \begin{align*}
    \sup_{n \geq 0} \left[P(X_{n} = 0)+P(X_{n}=1)\right] < 1,
  \end{align*}
which does not allow any generations where $P(X_{n}=1)$ is very close to 1, i.e. generations where almost
nothing happens. Hypothesis (H3) prevents the process from tending toward a pure birth process.
Hypotheses (H4a), an upper bounds on the first three moments, will be sufficient for convergence in distribution;
an upper bound on all moments, (H4b), will be necessary for convergence in moments.

Hypotheses (H5) and (H6) are our new criteria for the critical regime.
Hypothesis (H5) places the process out of the supercritical regime, into either the subcritical or critical regimes,
and hypothesis (H6) places the process out of the subcritical regime, into either the critical or supercritical regimes.
Hypothesis (H5) reduces to the familiar condition
  \begin{align*}
    \sum_{n=0}^{\infty} \frac{1}{E(Z_{n})} = \infty
  \end{align*}
when given control on the first two moments of the offspring distribution:
  \begin{align*}
    0 < \inf_{n \geq 0} E(X_{n}^{r}) \leq \sup_{n \geq 0}E(X_{n}^{r}) < \infty \qquad \text{for } r=1,2.
  \end{align*}
While it might be tempting to work with the tidier quantity
  \begin{align*}
    \sum_{j=0}^{n-1}\frac{1}{m_{j+1}} = \sum_{j=1}^{n}\frac{1}{E(Z_{j})},
  \end{align*}
the forthcoming result that $\Gamma_{n}\phi_{n} \to 1$ indicates that $\Gamma_{n}$ is the correct quantity
to consider.

On necessity: the process must go extinct with probability one for the conditioned, scaled process to tend
towards an exponential limit. Without this property, even a classical branching process will not exhibit
the same limiting behavior in the supercritical regime. Modulo some regularity conditions, hypothesis (H5)
ensures extinction. For hypothesis (H6), we rewrite the scaling factor via:
  \begin{align*}
    m_{n} & = E(Z_{n}) \\
    & = E(Z_{n} \mid Z_{n} = 0)P(Z_{n}=0)+E(Z_{n} \mid Z_{n} \ne 0)P(Z_{n} \ne 0) \\
    & = E(Z_{n} \mid Z_{n} \ne 0)\phi_{n}.
  \end{align*}
Thus, 
  \begin{align*}
    \frac{\zeta_{n}}{E(Z_{n} \mid Z_{n} \ne 0)} = \frac{\zeta_{n}}{m_{n}/\phi_{n}}.
  \end{align*}
Once we show $\phi_{n}\Gamma_{n} \to 1$, hypothesis (H6) will say that the denominator in the
conditioned, scaled process tends to infinity, which is certainly necessary if we are to take a discrete
process $\zeta_{n}$ to a continuous limit $\exp(1)$. Though we do not explore it here,
it is interesting to ask whether and under what hypotheses the conditioned, scaled process tends to
a geometric limit in an appropriately defined subcritical regime.

On improvement: compare (H5) and (H6) with Jagers' hypothesis
  \begin{align*}
    0 < \lim_{n \to \infty} E(Z_{n}) < \infty.
  \end{align*}
Jagers' hypothesis implies (H5) and (H6), and (H5) and (H6) do not imply Jagers' hypothesis.
Here we construct a simple example: let $P(X_{0}=1)=1$, and for $n \geq 1$ let
  \begin{align*}
    P(X_{n}=2) = \frac{1}{2}\frac{n+1}{n}, \qquad P(X_{n}=0) = 1-P(X_{n}=2).
  \end{align*}
Then $\mu_{n}=\frac{n+1}{n}$ and so $m_{n} = n$. This process satisfies (H1)-(H6), but
does not have bounded means. Examples can be constructed similarly to attain any
polynomial growth or decay of $m_{n}$, and such a process will satisfy (H5) and (H6) as
long as $m_{n}$ experiences polynomial decay, linear growth, or anything in between.

In addition to hypotheses (H1)-(H6), we will make use of an estimate due to Agresti \cite{Agresti}, Lemmas 1 and 2:
If $E(X_{n}^{2}) < \infty$ for each $n$ (which is weaker than (H4a) or (H4b)), then
  \begin{align*}
    \left(\frac{1}{(1-s)m_{n}} + \Gamma_{n}\right)^{-1} \leq 1-f_{n}(s) \leq 
      \left(\frac{1}{(1-s)m_{n}} + \sum_{j=0}^{n-1}\frac{g_{j}''(0)}{2\mu_{j}m_{j+1}}\right)^{-1}.
  \end{align*}
A sharper upper bound is permitted by Agresti's argument:
  \begin{equation}
    \left(\frac{1}{(1-s)m_{n}} + \Gamma_{n}\right)^{-1} \leq 1-f_{n}(s)
      \leq \left(\frac{1}{(1-s)m_{n}} + \sum_{j=0}^{n-1} \frac{g_{j}''(f_{jn}(0))}{2\mu_{j}m_{j+1}}\right)^{-1}.
  \end{equation}
Evaluating at $s=0$ leads to
  \begin{equation}
    \left(\frac{1}{m_{n}} + \Gamma_{n}\right)^{-1} \leq \phi_{n}
      \leq \left(\frac{1}{m_{n}} + \sum_{j=0}^{n-1} \frac{g_{j}''(f_{jn}(0))}{2\mu_{j}m_{j+1}}\right)^{-1}.
  \end{equation}

\subsection{Results} 
\label{Results}

The main theorem we'll prove is the following:

\begin{theorem}
  If (H1)-(H6) hold, then
    \begin{align*}
      \frac{\zeta_{n}}{E(\zeta_{n})} \to \exp(1),
    \end{align*}
  where the convergence is in distribution assuming (H4a), and in moments assuming (H4b).
\end{theorem}

We will prove Theorem 1 with (H4a) by showing that the moment generating function converges to that
of an exponential, and we will prove Theorem 1 with (H4b) by showing that
the $r$-th moment of the process tends to $r!$.
The proofs rely on two technical lemmas and three smaller theorems.

\begin{theorem}
  If (H1)-(H6) hold with (H4a), then extinction occurs with probability one.
\end{theorem}

Theorem 2 combines our hypotheses and (2) to bound $\Gamma_{n}\phi_{n}$ from above
and below, showing that $\phi_{n} \to 0$. The proof of Theorem 2 relies almost entirely on the two preceding
lemmas. Theorem 3, however, will require more careful estimates to show the asymptotic equality:

\begin{theorem}
  If (H1)-(H6) hold with (H4a), then $\phi_{n}\Gamma_{n} \to 1$.
\end{theorem}

Theorems 2 and 3 are sufficient to prove Theorem 1 with (H4a). For the moments convergence with (H4b),
Theorem 4 relates $\Gamma_{n}$ to the factorial moments of $Z_{n}$. 

\begin{theorem}
  If (H1)-(H6) hold with (H4b), then
    \begin{align*}
      \frac{F_{n,r}}{m_{n}^{r}\Gamma_{n}^{r-1}} \to r! \qquad (r \geq 1).
    \end{align*}
\end{theorem}

These are all the necessary parts to complete Theorem 1. As an endnote, we prove that
hypothesis (H6) can be substituted for a stronger but more intuitive condition:

\begin{theorem}
  Given (H1)-(H5), if $\mu_{n} \to 1$ then (H6) holds.
\end{theorem}

The difference between (H6) and the condition $\mu_{n} \to 1$ is that (H6) allows for generation
means that oscillate around 1. 

\subsection{Two Technical Lemmas} 
\label{Two Technical Lemmas}

First we point out a few immediate technical consequences of (H1)-(H6). Hypotheses (H1) and (H4)
imply that $\mu_{n}$ is uniformly bounded away from zero and from infinity. Hypothesis (H2) implies that
$\liminf_{n \to \infty} g_{n}''(1) > 0$, because $\sum_{k=2}^{\infty}P(X_{n}=k)$ is bounded away from zero.
And finally, because factorial moments can be expressed as
combinations of regular moments, we have $\sup_{n \geq 0}g_{n}^{(r)}(1) < \infty$ for $r=1,2,3$ under (H4a)
or all $r \geq 1$ under (H4b).

\begin{lemma}
  Hypotheses (H2) and (H4) imply that there exists an $N$ such that for every $0<\epsilon_{0} \leq 1$,
  we have
    \begin{align*}
      \inf_{\epsilon \geq \epsilon_{0}}\inf_{n \geq N} g_{n}''(\epsilon) > 0.
    \end{align*}
\end{lemma}

Lemma 1 says that the generating functions of the offspring distributions are convex, uniformly in $n$.
Convexity should follow intuitively from hypothesis (H2), because
  \begin{align*}
    g_{n}''(\epsilon) = \sum_{k=2}^{\infty}k(k-1)P(X_{n}=k)\epsilon^{k-2},
  \end{align*}
and (H2) tells us that the sum of the coefficients $\{P(X_{n}=k)\}_{k=2}^{\infty}$ stays away from zero.
To get uniformity in $n$, The only possible problem is mass tending towards the tails of the $X_{n}$ distribution,
and this is precluded by hypothesis (H4).

  \begin{proof}
    Since $g_{n}''(s)$ is a power series of nonnegative terms, it is increasing in $s$. Therefore the uniformity
    in $\epsilon \geq \epsilon_{0}$ follows immediately from showing that $\inf_{n \geq N} g_{n}''(\epsilon)>0$
    for some fixed $\epsilon>0$.
    
    By (H4) pick $M$ such that $\mu_{n} < M$ for every $n$. By (H2) pick $N,\delta>0$
    such that for all $n \geq N$, we have $\sum_{k=2}^{\infty} P(X_{n}=k) > \delta$.
    Then $\sum_{k=2}^{2M/\delta} P(X_{n}=k) \geq \frac{\delta}{2}$ for all $n \geq N$;
    to see this, suppose otherwise. Then there would be an $n \geq N$ such that
      \begin{align*}
        \mu_{n} \geq \sum_{k=2M/\delta}^{\infty} kP(X_{n}=k) \geq \frac{2M}{\delta}
          \sum_{k=2M/\delta}^{\infty}P(X_{n}=k) > M,
      \end{align*}
    a contradiction. Now we have
      \begin{align*}
        g_{n}''(\epsilon) = \sum_{k=2}^{\infty}k(k-1)P(X_{n}=k)\epsilon^{k-2}
          \geq \sum_{k=2}^{2M/\delta} P(X_{n}=k) \epsilon^{2M/\delta}
          \geq \frac{\delta}{2}\cdot \epsilon^{2M/\delta}.
      \end{align*}
  \end{proof}
  
\begin{lemma}
  Assumptions (H1)-(H5) imply that there exist $C,N > 0$ such that
  for every $n \geq N$, we have
    \begin{align*}
      \Gamma_{n} \equiv \sum_{j=0}^{n-1} \frac{g_{j}''(1)}{2\mu_{j}m_{j+1}}
         \leq C \sum_{j=0}^{n-1} \frac{g_{j}''(f_{jn}(0))}{2\mu_{j}m_{j+1}}.
    \end{align*}
\end{lemma}

Lemma 2 is a first, rough relation between the upper bound in (2) and $\Gamma_{n}$.
Note that $f_{jn}(0)$ is simply $1-\phi_{jn}$, the probability of the process started at time step $j$
dying out by time $n$. Because extinction is sure, we know that $\phi_{jn} \to 0$ as $n-j \to \infty$,
and so ``most'' of the terms of the sum on the right will have $f_{jn}(0)$ close to 1.

  \begin{proof}
    Hypothesis (H5) tells us that all of the mass of $\Gamma_{n}$ is in the tail of the sum, and so
    it suffices to show that there exist $N,C>0$ such that
    $g_{j}''(1) \leq Cg_{j}''(f_{jn}(0))$ for all $n \geq N$ and $j<n$.
    By (H3), fix $N$ such that $\inf_{n \geq N} g_{n}(0) > 0$. Now for any $j \geq N$,
    $f_{jn}$ is a composition of $g_{k}$'s satisfying $\inf_{k \geq N} ||g_{k}||_{\infty} > 0$,
    because every $g_{k}$ is an increasing function and bounded uniformly below at 0.
    Thus we have
      \begin{align*}
        \inf_{n > N} \inf_{N\leq j < n}f_{jn}(0) > 0.
      \end{align*}
    If, on the other hand, $j < N$, then we write
      \begin{align*}
        f_{jn}(0) = (f_{jN} \circ f_{Nn})(0).
      \end{align*}
    From the previous case, we know that $f_{Nn}(0)$ is bounded away from 0 uniformly. From (H1), we
    know that the $g_{j}$ are all strictly increasing and nonnegative, and so
    $f_{jN}$ is a composition of at most $N$ functions that are positive away from 0. Thus once again,
      \begin{align*}
        \inf_{n>N} \inf_{j < n} f_{jn}(0) > 0.
      \end{align*}
    Thus for $n \geq N$, the argument of $g_{j}''$ is bounded away from 0, and so
    by Lemma 1, picking $N$ larger if necessary, we have $\inf_{n \geq N}g_{j}''(f_{jn}(0)) > 0$.
    From (H4) we know that $\sup_{n \geq 0} g_{j}''(1) < \infty$, and so we can bound the latter quantity
    uniformly by the former. Since $\Gamma_{n} \to \infty$, we can ignore the first $N$ terms
    of the sum and substitute this bound in the sum, proving the desired result.
  \end{proof}

\subsection{Proof of Theorem 2} 
\label{Proof of Theorem 2}
To prove Theorem 2, we simply put together the lemmas we've shown, hypothesis (H6), and
the estimate (2).

  \begin{proof}
    Consider (2). Hypothesis (H6) tells us that the left hand side is asymptotically
    $\frac{1}{\Gamma_{n}}$. Hypothesis (H6) and Lemma 2 tell us that the right hand side
    is asymptotically smaller than $\frac{C}{\Gamma_{n}}$ for some constant $C$.
    Consequently,
      \begin{align*}
        1 \leq \liminf_{n \to \infty} \phi_{n} \Gamma_{n} \leq \limsup_{n \to \infty} \phi_{n} \Gamma_{n} \leq C.
      \end{align*}
    Since $\Gamma_{n} \to \infty$, we have $\phi_{n} \to 0$ and thus extinction occurs
    with probability one.
  \end{proof}

\subsection{Proof of Theorem 3} 
\label{Proof of Theorem 3}
The proof of Theorem 3 can be thought of as a more careful version of Lemma 2. In the sum
involving $g_{n}''(f_{jn}(0))$, we show that the terms in which $f_{jn}(0)$ is far away from 1
contribute a negligible amount to the sum, so that the constant $C$ in Lemma 2 is in fact equal to 1.

  \begin{proof}
  \begin{enumerate}[1.]
    \item Applying the inequality $\phi_{n} \leq \frac{C}{\Gamma_{n}}$ to the process
      started at time $j$ (that is, $\{Z_{n} | Z_{j} = 1\}$) gives us
        \begin{align*}
          \phi_{jn} \leq C \frac{1}{\sum_{k=j}^{n-1} \frac{g_{k}''(1)}{2\mu_{k}(m_{k+1}/m_{j})}}
            = C \frac{1}{m_{j}(\Gamma_{n}-\Gamma_{j})}.
        \end{align*}
      Note that $m_{k}/m_{j}$ is the mean of the process starting at time $j$.
    \item By (H3), pick $N$ large enough and $\eta>0$ small enough that $\inf_{n \geq N}g_{n}''(1)>\eta$.
      By (H4), pick $L$ such that $\sup_{n \geq 0}g_{n}'''(1)<L$. Let $\epsilon>0$ and define the index
        \begin{align*}
          J(n) = \min\{j : \phi_{jn} > \epsilon\eta/L\}.
        \end{align*}
      The indices satisfying $j<J(n)$ correspond to the terms of $\phi_{jn}$ that are very close
      to 0 (i.e. terms where $f_{jn}(0)$ is very close to 1). We will show that ``most'' terms satisfy
      $j<J(n)$, in the sense that we will show that the terms with $j \geq J(n)$ contribute
      negligibly to the summation in the upper bound of (2).
      
      Since extinction is sure, $\phi_{jn} \to 0$ for fixed $j$ as $n \to \infty$. Consequently,
      $J(n) \to \infty$ as $n \to \infty$. If $N \leq j < J(n)$, then $\phi_{jn} \leq \epsilon\eta/L$, and
      we calculate via the mean value theorem:
        \begin{align*}
          g_{j}''(f_{jn}(0)) & = g_{j}''(1-\phi_{jn}) \\
          & = g_{j}''(1)-g_{j}'''(\xi)\phi_{jn}, \qquad \xi \in (1-\phi_{jn},1) \\
          & \geq g_{j}''(1)-L\phi_{jn} \\
          & = g_{j}''(1)\left(1-\frac{L\phi_{jn}}{g_{j}''(1)}\right) \\
          & \geq g_{j}''(1)\left(1-\frac{L\phi_{jn}}{\eta}\right) \\
          & \geq g_{j}''(1)(1-\epsilon).
        \end{align*}
    \item Using (2) and our result in 2., we find at large $n$:
        \begin{align*}
          1 \geq \frac{1}{\phi_{n}\Gamma_{n}}
            & \geq \frac{1}{\Gamma_{n}}\left(\frac{1}{m_{n}}+\sum_{j=1}^{n-1}\frac{g_{j}''(f_{jn}(0))}{2\mu_{j}m_{j+1}}\right) \\
          & \geq \frac{1}{\Gamma_{n}}\sum_{j=1}^{J(n)-1}\frac{g_{j}''(f_{jn}(0))}{2\mu_{j}m_{j+1}} \\
          & \geq \frac{1}{\Gamma_{n}}\sum_{j=1}^{J(n)-1}\frac{g_{j}''(1)(1-\epsilon)}{2\mu_{j}m_{j+1}} \\
          & \geq \frac{1}{\Gamma_{n}}(1-\epsilon)\Gamma_{J(n)}.
        \end{align*}
      Although our result in 2. also requires $N \leq j$, we can ignore finitely many of the terms of the
      series because of hypothesis (H5). Next, using the fact that $\Gamma_{J(n)} \leq \Gamma_{n}$,
        \begin{align*}
          \frac{1}{\Gamma_{n}}(1-\epsilon)\Gamma_{J(n)}
          & = \frac{\Gamma_{J(n)}}{\Gamma_{n}}-\epsilon\frac{\Gamma_{J(n)}}{\Gamma_{n}} \\
          & \geq \frac{\Gamma_{J(n)}}{\Gamma_{n}}-\epsilon \\
          & = (1-\epsilon)-\frac{\Gamma_{n}-\Gamma_{J(n)}}{\Gamma_{n}}.
        \end{align*}
      We want to show that the rightmost term goes to zero. Our result in 1. and the definition of
      $J(n)$ give us
        \begin{align*}
          \frac{\Gamma_{n}-\Gamma_{J(n)}}{\Gamma_{n}} \leq \frac{C}{m_{J(n)}\phi_{J(n)n}\Gamma_{n}}
            \leq \frac{CL}{\epsilon\eta}\cdot\frac{1}{m_{J(n)}\Gamma_{n}}.
        \end{align*}
      Since $J(n) \to \infty$, we have
        \begin{align*}
          m_{J(n)}\Gamma_{n} \geq m_{J(n)}\Gamma_{J(n)} \to \infty,
        \end{align*}
      and thus the desired term goes to zero. Whence,
        \begin{align*}
          1 \geq \limsup_{n \to \infty} \frac{1}{\phi_{n}\Gamma_{n}} \geq
            \liminf_{n \to \infty} \frac{1}{\phi_{n}\Gamma_{n}} \geq 1-\epsilon,
        \end{align*}
      for every $\epsilon>0$, completing the proof.
  \end{enumerate}
  \end{proof}

\subsection{Proof of Theorem 4} 
\label{Proof of Theorem 4}
Now we move our attention to the factorial moments of $Z_{n}$, and their relation to $\Gamma_{n}$.
In addition to the presence of $\Gamma_{n}$ in (2), the proof of Theorem 4 offers some insight
into how $\Gamma_{n}$ arises as an important quantity. At this point, we start assuming (H4b).

  \begin{proof}
  \begin{enumerate}[1.]
    \item Taking the $r$-th derivative of $f_{n+1}(s) = f_{n}(g_{n}(s))$ and evaluating at $1$, we find
        \begin{align*}
          F_{n+1,r} = \mu_{n}^{r}F_{n,r}+C_{n,r-1}F_{n,r-1} +...+C_{n,1}F_{n,1},
        \end{align*}
      where $C_{n,j}$ is a coefficient composed of factorial moments of $X_{n}$ and by (H4b) is
      thus bounded uniformly in $n$ for each $j$. The only coefficient we'll need explicitly
      is $C_{n,r-1} = \binom{r}{2}g_{n}''(1)\mu_{n}^{r-2}$. Dividing the above quantity by $m_{n+1}$,
      we find a recursion relation for $F_{n,r}$ on $n$, which we then iterate:
        \begin{align*}
          \frac{F_{n+1,r}}{m_{n+1}^{r}} & = \frac{F_{n,r}}{m_{n}^{r}} +
            C_{n,r-1}\frac{F_{n,r-1}}{m_{n+1}^{r}} +...+ C_{n,1}\frac{F_{n,1}}{m_{n+1}^{r}} \\
          & = \frac{F_{n,r}}{m_{n}^{r}} +\sum_{k=1}^{r-1}\frac{C_{n,k}F_{n,k}}{m_{n+1}^{r}} \\
          & = \sum_{k=1}^{r-1}\sum_{j=0}^{n} \frac{C_{j,k}F_{j,k}}{m_{j+1}^{r}}.
        \end{align*}
      Then, expanding the sum over the other index and splitting off some of the factors of $m_{j+1}$
      in the denominator,
        \begin{align*}
          \frac{F_{n+1,r}}{m_{n+1}^{r}} & = \sum_{j=0}^{n}\sum_{k=1}^{r-1} \frac{C_{j,k}F_{j,k}}{m_{j+1}^{r}} \\
          & = \sum_{j=0}^{n} C_{j,r-1}\frac{F_{j,r-1}}{m_{j+1}^{r}}
            + ... + \sum_{j=0}^{n} C_{j,1}\frac{F_{j,1}}{m_{j+1}^{r}} \\
          & = \sum_{j=0}^{n}\frac{C_{j,r-1}}{m_{j}\mu_{j}^{r}}\frac{F_{j,r-1}}{m_{j}^{r-1}}
            + ... + \sum_{j=0}^{n}\frac{C_{j,1}}{m_{j}^{r-1}\mu_{j}^{r}}\frac{F_{j,1}}{m_{j}}.
        \tag{3}
        \end{align*}
    \item We will prove Theorem 4 inductively. We know that $F_{n,1}=m_{n}$, so the claim is clear
      for $r=1$. Assume inductively it's true for $k=1,...,r-1$, so that
        \begin{align*}
          \frac{F_{n,k}/m_{n}^{k}}{k!\Gamma_{n}^{k-1}} \overset{n \to \infty}{\longrightarrow} 1, \qquad (k=1,..., r-1),
        \end{align*}
      and also
        \begin{align*}
          \frac{\frac{C_{n,k}}{m_{n}^{r-k}\mu_{n}^{r}}F_{n,k}/m_{n}^{k}}
            {\frac{C_{n,k}}{m_{n}^{r-k}\mu_{n}^{r}}k!\Gamma_{n}^{k-1}} \overset{n \to \infty}{\longrightarrow} 1,
            \qquad (k=1,..., r-1).
        \end{align*}
      We aim to substitute this into (3).
      Recall the general fact that if $\frac{a_{n}}{b_{n}} \to 1$ and $b_{n}$ is not summable,
      then $\frac{\sum_{j=1}^{n}a_{j}}{\sum_{j=1}^{n}b_{j}} \to 1$ as well. We would like to use this fact to say
        \begin{align*}
          \frac{\sum_{j=0}^{n}\frac{C_{j,k}}{m_{j}^{r-k}\mu_{j}^{r}} F_{j,k}/m_{j}^{k}}
            {\sum_{j=0}^{n} \frac{C_{j,k}}{m_{j}^{r-k}\mu_{j}^{r}}k!\Gamma_{j}^{k-1}} 
            \to 1, \qquad (k=1,...,r-1),
        \end{align*}
      but we can only use
      this relation in the largest few values of $k$, because
      $\frac{C_{n,k}}{m_{n}^{r-k}\mu_{n}^{r}}k!\Gamma_{n}^{k-1}$ is not summable for $k=r-1$
      (corresponding to the leading series in (3)) and it is summable for $k=1$ (corresponding
      to the last series in (3)). But since we're adding all these sums together, we only
      care about the ones that tend to infinity, so we'll clump the finite sums into an $O(1)$, and use
      the fact for the rest:
        \begin{align*}
          \frac{F_{n+1,r}}{m_{n+1}^{r}}& = \left[\sum_{j=0}^{n}\frac{C_{j,r-1}}{m_{j}\mu_{j}^{r}}
            (r-1)!\Gamma_{j}^{r-2} + o\left(\sum_{j=0}^{n}\frac{C_{j,r-1}}{m_{j}\mu_{j}^{r}}
            (r-1)!\Gamma_{j}^{r-2} \right)\right] \\
            & \qquad + \left[\sum_{j=0}^{n}\frac{C_{j,r-2}}{m_{j}\mu_{j}^{r}}(r-2)!\Gamma_{j}^{r-3}
              + o\left(\sum_{j=0}^{n}\frac{C_{j,r-2}}{m_{j}\mu_{j}^{r}}(r-2)!\Gamma_{j}^{r-3}\right)\right]
              + ... + O(1) \\
          & = \left[\sum_{j=0}^{n}\frac{(r-1)!C_{j,r-1}}{\mu_{j}^{r}} \frac{\Gamma_{j}^{r-2}}{m_{j}}
            + o\left(\sum_{j=0}^{n}\frac{\Gamma_{j}^{r-2}}{m_{j}} \right) \right] \\
            & \qquad + \left[\sum_{j=0}^{n}\frac{(r-2)!C_{j,r-2}}{\mu_{j}^{r}} \frac{\Gamma_{j}^{r-3}}{m_{j}^{2}}
            + o\left(\sum_{j=0}^{n}\frac{\Gamma_{j}^{r-3}}{m_{j}^{2}} \right) \right] + ... + O(1).
        \end{align*}
    \item The growth of the above sums is driven by the $\frac{\Gamma_{j}}{m_{j}}$ term.
      Since $m_{n}\Gamma_{n} \to \infty$, the first sum is of higher order than the others, and thus
        \begin{align*}
          \frac{F_{n+1,r}}{m_{n+1}^{r}}
            & = \sum_{j=0}^{n}\frac{(r-1)!C_{j,r-1}}{\mu_{j}^{r}} \frac{\Gamma_{j}^{r-2}}{m_{j}}
            + o\left(\sum_{j=0}^{n}\frac{\Gamma_{j}^{r-2}}{m_{j}} \right) \\
          & = r!\sum_{j=0}^{n}\frac{(r-1)g_{j}''(1)}{2\mu_{j}}\frac{\Gamma_{j}^{r-2}}{m_{j+1}}
            + o\left(\sum_{j=0}^{n} \frac{\Gamma_{j}^{r-2}}{m_{j}}\right),
        \end{align*}
      where in the last line we've substituted for $C_{j,r-1}$.
      To complete the inductive proof we just need to show
        \begin{align*}
          \Gamma_{n+1}^{r-1} = \sum_{j=0}^{n} \frac{(r-1)g_{j}''(1)}{2\mu_{j}}
            \frac{\Gamma_{j}^{r-2}}{m_{j+1}} + o\left(\sum_{j=0}^{n}\frac{\Gamma_{j}^{r-2}}{m_{j+1}}\right).
        \end{align*}
      To simplify notation, we show it for $r \mapsto r+1$; that is, we want to show
        \begin{align*}
          \Gamma_{n+1}^{r} = \sum_{j=0}^{n} \frac{rg_{j}''(1)}{2\mu_{j}}
            \frac{\Gamma_{j}^{r-1}}{m_{j+1}} + o\left(\sum_{j=0}^{n}\frac{\Gamma_{j}^{r-1}}{m_{j+1}}\right).
        \end{align*}
    \item A binomial expansion of $\Gamma_{n+1}^{r}$ gives us another recurrence relation:
        \begin{align*}
          \Gamma_{n+1}^{r} & = \left(\Gamma_{n}+\frac{g_{n}''(1)}{2\mu_{n}m_{n+1}}\right)^{r} \\
          & = \sum_{j=0}^{r} \binom{r}{j}\left(\frac{g_{n}''(1)}{2\mu_{n}m_{n+1}}\right)^{j}
            \Gamma_{n}^{r-j} \\
          & = \Gamma_{n}^{r}+r\frac{g_{n}''(1)}{2\mu_{n}}\frac{\Gamma_{n}^{r-1}}{m_{n+1}}
            + ... + \left(\frac{g_{n}''(1)}{2\mu_{n}}\right)^{r}\frac{1}{m_{n+1}^{r}}.
        \end{align*}
      Expanding that recurrence relation yields
        \begin{align*}
          \Gamma_{n+1}^{r} & = \sum_{j=0}^{n}\frac{rg_{j}''(1)}{2\mu_{j}}\frac{\Gamma_{j}^{r-1}}{m_{j+1}}
            + ... + \sum_{j=0}^{n}\left(\frac{g_{j}''(1)}{2\mu_{j}}\right)^{r}\frac{1}{m_{j+1}^{r}}.
        \end{align*}
      As before, the first sum dominates the others, and we have
        \begin{align*}
          \Gamma_{n+1}^{r} & = \sum_{j=0}^{n}\frac{rg_{j}''(1)}{2\mu_{j}}
            \frac{\Gamma_{j}^{r-1}}{m_{j+1}} + o\left(\sum_{j=0}^{n}\frac{\Gamma_{j}^{r-1}}{m_{j+1}}\right),
        \end{align*}
      as desired.
  \end{enumerate}
  \end{proof}

\subsection{Proof of Theorem 1} 
\label{Proof of Theorem 1}
Proving Theorem 1 is now just a matter of putting together the pieces of the other theorems.

  \begin{proof}
  \begin{enumerate}[1.]
    \item Assume (H4a). We aim to show that the moment generating function of
      $\frac{\zeta_{n}}{E(\zeta_{n})}$ converges to $\frac{1}{1-s}$. First we expand
      the generating function of $Z_{n}$ by:
        \begin{align*}
          f_{n}(s) = E(s^{Z_{n}}) & = P(Z_{n} \ne 0)E(s^{Z_{n}} | Z_{n} \ne 0)+P(Z_{n} = 0)E(s^{Z_{n}} | Z_{n} =0) \\
          & = \phi_{n}E(s^{Z_{n}} | Z_{n} \ne 0)+(1-\phi_{n}).
        \end{align*}
      Then we substitute $s= \exp(-\frac{s}{E(Z_{n}|Z_{n} \ne 0)})$ to obtain the moment generating function
      of $-s$, which we will show tends to $\frac{1}{1+s}$. Also substituting $E(Z_{n}|Z_{n} \ne 0) = m_{n}/\phi_{n}$,
        \begin{align*}
          E\left(\exp\left(-\frac{sZ_{n}}{E(Z_{n}|Z_{n}\ne 0)}\right) \middle| Z_{n} \ne 0\right)
            & = \frac{f_{n}(\exp(-\frac{s}{E(Z_{n}|Z_{n} \ne 0)}))-1}{\phi_{n}}+1 \\
            & = \frac{f_{n}(\exp(-\frac{s\phi_{n}}{m_{n}}))-1}{\phi_{n}}+1.
        \end{align*}
    \item Part 3 of the proof of Theorem 3 showed that
        \begin{align*}
          \sum_{j=0}^{n-1}\frac{g_{j}''(f_{jn}(0))}{2\mu_{j}m_{j+1}} = \Gamma_{n}+o(\Gamma_{n}),
        \end{align*}
      and thus (1) takes the form
        \begin{align*}
          \left(\frac{1}{(1-s)m_{n}} + \Gamma_{n}\right)^{-1} \leq 1-f_{n}(s)
            \leq \left(\frac{1}{(1-s)m_{n}} + \Gamma_{n}+o(\Gamma_{n}) \right)^{-1}.
        \end{align*}
      Cleaning up this expression to make it look more like our moment generating function, we find
        \begin{align*}
          1-\frac{1}{\phi_{n}}\left(\frac{1}{(1-s)m_{n}} + \Gamma_{n}+o(\Gamma_{n})\right)^{-1}
            \leq \frac{f_{n}(s)-1}{\phi_{n}}+1
            \leq 1-\frac{1}{\phi_{n}}\left(\frac{1}{(1-s)m_{n}} + \Gamma_{n} \right)^{-1}.
        \end{align*}
      And substituting $s = \exp(-\frac{s\phi_{n}}{m_{n}})$,
        \begin{align*}
          1-\frac{1}{\phi_{n}}\left(\frac{1}{(1-\exp(-\frac{s\phi_{n}}{m_{n}}))m_{n}}
            +\Gamma_{n}+o(\Gamma_{n})\right)^{-1}
            & \leq E\left(\exp\left(-\frac{sZ_{n}}{E(Z_{n}|Z_{n}\ne 0)}\right) \middle| Z_{n} \ne 0\right) \\
            & \qquad \leq 1-\frac{1}{\phi_{n}}\left(\frac{1}{(1-\exp(-\frac{s\phi_{n}}{m_{n}}))m_{n}}
            +\Gamma_{n}\right)^{-1}.
        \end{align*}
      Thus it suffices to show that
        \begin{align*}
          1-\frac{1}{\phi_{n}}\left(\frac{1}{(1-e^{-\frac{s\phi_{n}}{m_{n}}})m_{n}}+\Gamma_{n}\right)^{-1}
            \to \frac{1}{1+s}.
        \end{align*}
    \item Hypothesis (H6) and Theorem 3 imply that $\frac{\phi_{n}}{m_{n}} \to 0$,
      and so $e^{-\frac{s\phi_{n}}{m_{n}}} \to 1$. Taking a first order approximation
      and using Theorem 3,
        \begin{align*}
          1-\frac{1}{\phi_{n}}\left(\frac{1}{(1-e^{-\frac{s\phi_{n}}{m_{n}}})m_{n}}+\Gamma_{n}\right)^{-1}
            & = 1-\frac{1}{\phi_{n}}\left(\frac{1}{(s\phi_{n}/m_{n}+o(\phi_{n}/m_{n}))m_{n}}+\Gamma_{n}
              \right)^{-1} \\
          & = 1-\frac{1}{\Gamma_{n}\phi_{n}}\left(\frac{1}{s\phi_{n}\Gamma_{n}+o(\phi_{n}\Gamma_{n})}
            + 1 \right)^{-1} \\
          & \to 1-\frac{s}{1+s} \\
          & = \frac{1}{1+s},
        \end{align*}
      as desired. This concludes the proof in the case of (H4a).
    \item Now assume (H4b). The $r$-th moment of the conditioned, scaled process can be written as
        \begin{align*}
          E\left(\left(\frac{Z_{n}}{E(Z_{n}|Z_{n} \ne 0)}\right)^{r}\middle|Z_{n} \ne 0\right)
            & = \frac{E(Z_{n}^{r}|Z_{n} \ne 0)}{E(Z_{n}|Z_{n} \ne 0)^{r}} \\
          & = \frac{M_{n,r}/\phi_{n}}{(m_{n}/\phi_{n})^{r}} \\
          & = \frac{M_{n,r}\phi_{n}^{r-1}}{m_{n}^{r}}.
        \end{align*}
      Theorem 3 tells us
        \begin{align*}
          \lim_{n \to \infty} \frac{M_{n,r}\phi_{n}^{r-1}}{m_{n}^{r}}
            & = \lim_{n \to \infty} \frac{M_{n,r}}{m_{n}^{r}\Gamma_{n}^{r-1}}.
        \end{align*}
      We can write moments in terms of factorial moments by
        \begin{align*}
          M_{n,r} = \sum_{k=0}^{r} \left\{\begin{matrix}r \\ k \end{matrix}\right \}F_{n,k},
        \end{align*}
      where $\{\begin{smallmatrix}r\\k\end{smallmatrix}\}$ is a Stirling number of the second kind.
      Theorem 4 and hypothesis (H5) tell us that the highest-index term dominates the
      sum, and since $\{\begin{smallmatrix} r \\ r \end{smallmatrix}\} = 1$, we find
        \begin{align*}
          \lim_{n \to \infty} \frac{M_{n,r}}{m_{n}^{r}\Gamma_{n}^{r-1}}
            = \lim_{n \to \infty} \frac{F_{n,r}}{m_{n}^{r}\Gamma_{n}^{r-1}} = r!
        \end{align*}
      as desired, completing the proof.
  \end{enumerate}
  \end{proof}

\subsection{Proof of Theorem 5} 
\label{Proof of Theorem 5}

  \begin{proof}
    Hypotheses (H1)-(H4) let us write
      \begin{align*}
        \Gamma_{n} C \geq \sum_{j=0}^{n-1}\frac{1}{m_{j+1}} \geq C'\Gamma_{n}.
      \end{align*}
    for constants $C,C'$. Since $m_{j+1}=\mu_{j}m_{j}$, we can write
      \begin{align*}
        \frac{1}{m_{n}} = 1+\sum_{j=0}^{n-1} \left( \frac{1}{m_{j+1}}-\frac{1}{m_{j}}\right)
          = 1+\sum_{j=0}^{n}\frac{1-\mu_{j}}{m_{j+1}}.
      \end{align*}
    Fix $\epsilon>0$ and pick $N$ large enough that $|1-\mu_{n}|<\epsilon$
    for all $n > N$. Then when $n > N$, combining these inequalities yields
      \begin{align*}
        \frac{1}{m_{n}\Gamma_{n}} & \leq C \dfrac{1+\sum_{j=0}^{n-1}\frac{|1-\mu_{j}|}{m_{j+1}}}
          {\sum_{j=0}^{n-1}\frac{1}{m_{j+1}}} \\
        & = \frac{C}{\sum_{j=0}^{n-1}\frac{1}{m_{j+1}}}+\frac{C\sum_{j=0}^{N}
          \frac{|1-\mu_{j}|}{m_{j+1}}}{\sum_{j=0}^{n-1}\frac{1}{m_{j+1}}}
          + \frac{C\sum_{j=N+1}^{n-1}\frac{|1-\mu_{j}|}{m_{j+1}}}
          {\sum_{j=0}^{n-1}\frac{1}{m_{j+1}}} \\
        & \leq \frac{C}{\sum_{j=0}^{n-1}\frac{1}{m_{j+1}}}+\frac{C\sum_{j=0}^{N}
          \frac{|1-\mu_{j}|}{m_{j+1}}}{\sum_{j=0}^{n-1}\frac{1}{m_{j+1}}}
          + C\epsilon.
      \end{align*}
    Hypothesis (H5) tells us that the first two sums tend to zero, and thus the entire last expression
    tends to $C\epsilon$. Thus $\frac{1}{m_{n}\Gamma_{n}} \to 0$, as desired.
  \end{proof}

This concludes the discrete setting of branching processes. Now we move on to the continuous
setting of birth-and-death processes.

\section{The Continuous Setting} 
\label{The Continuous Setting}

The continuous setting will be presented and proved in exactly the same manner that
the discrete setting was. However, the discrete setting relied heavily on generating function
relations that are not present in the continuous setting, so the continuous setting will
require different techniques.

\subsection{Notation} 
\label{Notation2}
We will often drop the time argument of rates and means when it is implied.

We define the function $X_{t}$ by
  \begin{align*}
    -1 & \mapsto b_{-1}(t) \\
    1 & \mapsto b_{1}(t) \\
    & \vdots \\
    n & \mapsto b_{n}(t).
  \end{align*}
In this sense, $X_{t}$ is similar to a random variable that has probability $b_{k}(t)$ of taking value $k$;
however, $\sum_{k=-1,1,...,n} b_{k}(t)$ need not equal 1, so $X_{t}$ is not an actual random variable.
For purely notational reasons, though, it will be useful to denote moments of $X_{t}$ as though
it were a random variable; thus we denote
  \begin{align*}
    E(X_{t}) & = \sum_{k=-1,1,\dots,n}kb_{k}(t) \\
    E(X_{t}^{r}) & = \sum_{k=-1,1,\dots,n}k^{r}b_{k}(t).
  \end{align*}
We define $b_{0} = -\sum_{k=-1,1,...,n} b_{k}$, which gives us a cleaner expression for
transition probabilities:
  \begin{align*}
  \begin{cases}
    P(Z_{t+\Delta}-Z_{t}=k \mid Z_{t}) = Z_{t}\Delta b_{k}(t) + o(\Delta), & (k=-1,1,...,n) \\
    P(Z_{t+\Delta}-Z_{t}=0 \mid Z_{t}) = 1+Z_{t}\Delta b_{0}(t) + o(\Delta).
  \end{cases}
  \end{align*}
In the same way that $X_{t}$ is an analogue of $X_{n}$ in the discrete case,
we introduce the function
  \begin{align*}
    g_{t}(x) = \sum_{k=-1}^{n} b_{k}(t)x^{k+1},
  \end{align*}
which is (similarly to $X_{t}$) not a generating function, because the coefficient $b_{0}$ is negative.
It will nonetheless prove a useful notation for our needs. Indeed, we remark immediately that
  \begin{align*}
    g_{t}(1) & = 0 \\
    g_{t}'(1) & = E(X_{t}) \\
    g_{t}''(1) & = E(X_{t}^{2})+E(X_{t}) \\
    g_{t}'''(1) & = E(X_{t}^{3})-E(X_{t}).
  \end{align*}
We denote the moments of $Z_{t}$ by
  \begin{align*}
    M(t) & = E(Z_{t}) \\
    M_{r}(t) & = E(Z_{t}^{r})
  \end{align*}
and probabilities of survival by
  \begin{align*}
    \phi(t) & = P(Z_{t} \ne 0) \\
    \phi(s,t) & = P(Z_{t} \ne 0 | Z_{s} = 1).
  \end{align*}
Once again, we will denote by $\zeta_{t}$ the process conditioned on non-extinction:
  \begin{align*}
    \zeta_{t} = \{Z_{t} \mid Z_{t} \ne 0\}.
  \end{align*}
Finally, we introduce our continuous analogue of $\Gamma$:
  \begin{align*}
    \Gamma(t) = \int_{0}^{t}\frac{E(X_{s}^{2})}{2M(s)}ds,
  \end{align*}
which will once again be shown to satisfy $\Gamma(t)\phi(t) \to 1$ in the critical regime.

\subsection{Hypotheses} 
\label{Hypotheses2}
We'll make use of the following hypotheses:
\begin{enumerate}
  \item[(H7)] $\liminf_{t \geq 0} b_{-1}(t)  > 0$
  \item[(H8)] $\liminf_{t \geq 0} b_{1}(t)+...+b_{n}(t) > 0$
  \item[(H9)] $\sup_{t\geq 0} [b_{-1}(t)+b_{1}(t)+...+b_{n}(t)] < \infty$
  \item[(H10)] $\Gamma(t) \to \infty$
  \item[(H11)] $M(t)\Gamma(t) \to \infty$.
\end{enumerate}

Hypotheses (H7)-(H9) are the analogous regularity conditions.
Hypotheses (H7) and (H8) prevent the process from tending towards a pure birth process or pure death
process, respectively.

Hypothesis (H9) bounds all moments of $X_{t}$, uniformly in $t$. Here we see an
implication of allowing only finitely many births: there is no option to bound just
finitely many moments and prove that the conditioned, scaled process converges
in distribution. Indeed, our analogue to Theorem 1 will only deal with convergence in moments.

Hypothesis (H9) also assures us the ability to approximate the $b_{k}$ by
continuous functions with arbitrarily small $L^{1}$ error. Lemma 1 will show
that we only need to care about the $L^{1}$ behavior of the rates and not the
pointwise behavior. Consequently, we will often treat the $b_{k}$ as though
they were continuous to differentiate integrals involving the $b_{k}$.

Hypotheses (H10) and (H11) are exactly analogous to the discrete setting: hypothesis (H10)
places the process out of the supercritical regime, and hypothesis (H11) places the process
out of the subcritical regime. Under our regularity assumptions, hypothesis (H10) is
actually equivalent to the familiar condition
  \begin{align*}
    \int_{0}^{\infty}\frac{1}{E(Z_{t})}dt = \infty.
  \end{align*}
Hypotheses (H10) and (H11) are necessary for exactly the same reasons that they
are necessary in the discrete setting.

In addition to hypotheses (H7)-(H11), we will make use of an estimate due to Cistjakov and Markova
\cite{CM}, Theorem 1:
  \begin{align*}
    \phi(t) = \dfrac{1}{\dfrac{1}{M(t)}+\displaystyle\int_{0}^{t}\dfrac{g_{s}(1-\phi(s,t))+\phi(s,t)g_{s}'(1)}
      {\phi(s,t)^{2}M(s)}ds}.
  \end{align*}
A second order expansion of $g_{s}$ (a sharper version of \cite{CM}, Theorem 2) yields
  \begin{equation}
    \frac{1}{\dfrac{1}{M(t)}+\displaystyle\int_{0}^{t}\dfrac{g_{s}''(1)}{2M(s)}ds}
      \leq \phi(t) \leq \frac{1}{\dfrac{1}{M(t)}+\displaystyle\int_{0}^{t}\dfrac{g_{s}''(1-\phi(s,t))}{2M(s)}ds}.
  \tag{4}
  \end{equation}
  
\subsection{Results} 
\label{Results2}
The main theorem we'll prove is the following:

\begin{theorem}
  If (H7)-(H11) hold, then
    \begin{align*}
      \frac{\zeta_{t}}{E(\zeta_{t})} \to \exp(1),
    \end{align*}
  where the convergence is in moments.
\end{theorem}

The proof relies on four lemmas and three smaller theorems.

\begin{theorem}
  If (H7)-(H11) hold, then extinction occurs with probability one.
\end{theorem}

Theorem 7 will be proven in the same way that Theorem 2 was in the discrete setting,
and likewise, Theorem 8 will be proven as Theorem 3 was.

\begin{theorem}
  If (H7)-(H11) hold, then $\phi(t)\Gamma(t) \to 1$.
\end{theorem}

Whereas Theorem 4 was proven using recurrence relations between generating
functions, Theorem 9 will be proven using recurrence relations between the moments of $Z_{t}$;
specifically, the relations in the forthcoming Lemma 3.

\begin{theorem}
  If (H7)-(H11) hold, then
    \begin{align*}
      \frac{M_{r}(t)}{M(t)^{r}\Gamma(t)^{r-1}} \to r! \qquad (r \geq 1).
    \end{align*}
\end{theorem}

Theorems 7-9 will be enough to prove Theorem 6. Once again, we end with
an additional theorem that allows us to substitute hypothesis (H11) for a stronger
but more intuitive condition:

\begin{theorem}
  Given (H7)-(H10), if $E(X_{t}) \to 0$ then (H11) holds.
\end{theorem}

\subsection{Four Lemmas} 
\label{Four Lemmas}

\begin{lemma}
  \begin{align*}
    \dot{M}_{r}(t) = \sum_{j=1}^{r}\binom{r}{j}M_{r-j+1}(t)E(X_{t}^{j}).
  \end{align*}
\end{lemma}

This recurrence relation between the moments of the process will replace the
recurrence relation between generating functions in the discrete setting.
Lemma 3 implies in particular that  $M(t) = \exp(\int_{0}^{t}E(X_{s})ds)$, an equality
we used when transcribing (4).

  \begin{proof}
    The probability of two changes in $Z_{t}$ happening in a time span of $\Delta$ is
    $o(\Delta)$, so we calculate
      \begin{align*}
        E(Z_{t+\Delta}^{r}|Z_{t}) & = \sum_{k=-1}^{n}(Z_{t}+k)^{r}P(Z_{t+\Delta}=Z_{t}+k \mid Z_{t})
          + o(\Delta) \\
        & = Z_{t}^{r}(1+Z_{t}\Delta b_{0})\sum_{k=-1,1,...,n} (Z_{t}+k)^{r}Z_{t}\Delta b_{k} + o(\Delta) \\
        &= Z_{t}+ Z_{t}\Delta\sum_{k=-1}^{n}(Z_{t}+k)^{r} b_{k} + o(\Delta).
      \end{align*}
    Then, rearranging and sending $\Delta \to 0$,
      \begin{align*}
        \frac{E(Z_{t+\Delta}^{r}-Z_{t}^{r}\mid Z_{t})}{\Delta} & = Z_{t}\sum_{k=-1}^{n}(Z_{t}+k)^{r}b_{k}+o(\Delta) \\
        \dot{M}_{r} \mid Z_{t} & = Z_{t}\sum_{k=-1}^{n}(Z_{t}+k)^{r}b_{k} \\
        & = \sum_{k=-1}^{n} b_{k}\sum_{j=0}^{r}\binom{r}{j}Z_{t}^{r-j+1}k^{j}.
      \end{align*}
    The entire $k=0$ term, which is just $b_{0}Z_{t}^{r+1}$, cancels all of the $j=0$ terms from
    the other choices of $k$, so that we're left with
      \begin{align*}
        \dot{M}_{r} \mid Z_{t} & = \sum_{k=-1,1,...,n} b_{k}\sum_{j=1}^{r}\binom{r}{j}Z_{t}^{r-j+1}k^{j}.
      \end{align*}
    Taking expectations of both sides,
      \begin{align*}
        \dot{M}_{r} & = \sum_{k=-1,1,...,n} b_{k}\sum_{j=1}^{r}\binom{r}{j}M_{r-j+1}k^{j} \\
        & = \sum_{j=1}^{r}\binom{r}{j}M_{r-j+1}\sum_{k=-1,1,...,n}b_{k}k^{j} \\
        & = \sum_{j=1}^{r}\binom{r}{j}M_{r-j+1}E(X_{t}^{j}),
      \end{align*}
    as desired.
  \end{proof}

\begin{lemma}
  Hypotheses (H10) and (H11) imply that $\Gamma(t)$ is asymptotically equal to
    \begin{align*}
      \int_{0}^{t}\frac{g_{s}''(1)}{2M(s)}ds.
    \end{align*}
\end{lemma}

  \begin{proof}
    Using Lemma 3, we calculate
      \begin{align*}
        \int_{0}^{t}\frac{g_{s}''(1)}{2M(s)} & = \int_{0}^{t} \dfrac{E(X_{s}^{2})+E(X_{s})}{2M(s)} ds \\
        & = \Gamma(t) - \int_{0}^{t} \frac{-E(X_{s})}{2M(s)} ds \\
        & = \Gamma(t) - \int_{0}^{t}\frac{d}{ds}\frac{1}{2M(s)}ds \\
        & = \Gamma(t) - \frac{1}{2M(t)} + \frac{1}{2}.
      \end{align*}
    Hypotheses (H10) and (H11) tell us that this last expression is dominated by the
    $\Gamma(t)$ term at large times, as desired.
  \end{proof}

\begin{lemma}
  Hypotheses (H7) and (H9) imply that there exist $T,\eta>0$ such that for all $s \geq T$
  and $t \geq s+1$, we have $1-\phi(s,t) \geq \eta$.
\end{lemma}

Lemma 5 says that once we reach a point where hypothesis (H7) kicks in, we
expect to have some fixed minimal probability of extinction after a fixed time increment.

  \begin{proof}
    For any fixed $s$, we know that $\phi(s,t)$ is increasing in $t$. Thus it suffices to prove
    the claim for $t=s+1$. By hypothesis (H7), pick $T$ large enough that
      \begin{align*}
        C \equiv \inf_{x \geq T} b_{-1}(x)
      \end{align*}
    is nonzero, and let $D \equiv \sup_{x \geq T} b_{1}(x)+...+b_{n}(t)$, which by hypothesis
    (H9) is finite. Define a new birth-and-death process $\tilde{Z}_{t}$ with constant death rate
    $\tilde{b}_{-1} = C$ and birth rates $\tilde{b}_{1} = ... = \tilde{b}_{n} = D$. Then
    for all times greater than $T$, $\tilde{Z}_{t}$ has a smaller death rate and higher birth rates
    than $Z_{t}$, and so the corresponding probability of survival $\tilde{\phi}(s,t)$ is smaller than
    $\phi(s,t)$. Since the rates of $\tilde{Z}_{t}$ are constant, we know that
    $\tilde{\phi}(s,t) = \tilde{\phi}(t-s)$ depends only on the time elapsed, and so picking
    $\eta = 1-\tilde{\phi}(1)$ gives our desired uniform bound.
  \end{proof}
  
\begin{lemma}
  Hypotheses (H9) and (H11) implies that $\Gamma(t)-\Gamma(t-1)$ is $o(\Gamma(t))$. 
\end{lemma}

Lemma 6 will be used to take care of some dangling terms from requiring $t \geq s+1$ in Lemma 5.

  \begin{proof}
    A first order approximation of $\Gamma$ gives us
      \begin{align*}
        \Gamma(t) & = \Gamma(t-1) + \frac{E(X_{\xi}^{2})}{2M(\xi)}, \qquad \xi \in (t-1,t) \\
        & \leq \Gamma(t-1) + \frac{C}{M(\xi)}.
      \end{align*}
    Hypothesis (H11) implies that $\frac{C}{M(\xi)}$ is $o(\Gamma(t))$, and the result follows.
  \end{proof}
  
\subsection{Proof of Theorem 7} 
\label{Proof of Theorem 7}

As in Theorem 2, Theorem 7 will rely on (4), the preceding lemmas, and a rough
lower bound on $g_{s}''(1-\phi(s,t))$.

  \begin{proof}
  \begin{enumerate}[1.]
    \item Using Lemma 4 and hypothesis (H11) with (4) gives us
        \begin{align*}
          1 \leq \liminf_{t \geq 0} \phi(t)\Gamma(t).
        \end{align*}
    \item We expand $g_{s}''(1-\phi(s,t))$ by its definition:
        \begin{align*}
          g_{s}''(1-\phi(s,t)) = 2b_{1}(s)+6b_{2}(s)(1-\phi(s,t))+...+n(n+1)b_{n}(s)(1-\phi(s,t))^{n-1}.
        \end{align*}
      Consequently, Lemma 5 and hypothesis (H8) imply that there exists a $T$ large enough that
        \begin{align*}
          \inf_{s \geq T} g_{s}''(1-\phi(s,t)) > 0
        \end{align*}
      for $t \geq s+1$. Hypothesis (H9) also implies
        \begin{align*}
          \sup_{s \geq 0} g_{s}''(1-\phi(s,t)) < \infty.
        \end{align*}
      Using hypothesis (H10), Lemma 6, hypothesis (H9), and the lower bound on
      $g_{s}''(1-\phi(s,t))$ in that order, we find for $t>>T$,
        \begin{align*}
          \Gamma(t) & \leq C \int_{T}^{t}\frac{E(X_{s}^{2})}{2M(s)}ds \\
          & \leq C \int_{T}^{t-1}\frac{E(X_{s}^{2})}{2M(s)}ds \\
          & \leq C \int_{T}^{t-1}\frac{g_{s}''(1-\phi(s,t))}{2M(s)}ds \\
          & \leq C \int_{0}^{t} \frac{g_{s}''(1-\phi(s,t))}{2M(s)}ds.
        \end{align*}
      Inserting this into (4) and using hypothesis (H11), we find
        \begin{align*}
          \limsup_{t \geq 0}\phi(t)\Gamma(t) \leq \limsup_{t \geq 0}\frac{C}{\dfrac{1}{M(t)\Gamma(t)}+1} = C.
        \end{align*}
      Finally, we combine this inequality with the one from 1. to find
        \begin{align*}
          1 \leq \liminf_{t \geq 0} \phi(t)\Gamma(t) \leq \limsup_{t \geq 0}\phi(t)\Gamma(t) \leq C.
        \end{align*}
      Hypothesis (H10) then implies that $\phi(t) \to 0$.
  \end{enumerate}
  \end{proof}

\subsection{Proof of Theorem 8} 
\label{Proof of Theorem 8}

As with Theorem 3, Theorem 8 relies on a more careful estimate of $g_{s}''(1-\phi(s,t))$ by
splitting the sum in (4) into regions where $\phi(s,t)$ is close to 0 and far away from 0.
Once again, we will show that the region where $\phi(s,t)$ is far away from 0 contributes
negligibly to the integral.

  \begin{proof}
  \begin{enumerate}[1.]
    \item Applying the inequality $\phi(t) \leq \frac{C}{\Gamma(t)}$ to the process started
      at time $s$ gives us
        \begin{align*}
          \phi(s,t) \leq C\frac{1}{\int_{s}^{t}\frac{E(X_{u}^{2})}{2M(u)/M(s)}du}
            = C\frac{1}{M(s)(\Gamma(t)-\Gamma(s))}.
        \end{align*}
      Here we used the fact that $M(u)/M(s)$ is the mean of the process starting at time $s$.
    \item By hypothesis (H9), pick $L$ such that $g_{t}'''(1) \leq L$ for all $t$.
      By hypothesis (H8), pick $T,\eta>0$
      such that $\inf_{t \geq T} g_{t}''(1) > \eta$. Let $\epsilon>0$ and define
        \begin{align*}
          S(t) = \min\{s : \phi(s,t) \geq \epsilon\eta/L\}.
        \end{align*}
      $S(t)$ is the time that divides our regions where $\phi(s,t)$ is small (when $s<S(t)$),
      and where $\phi(s,t)$ is large (when $s \geq S(t)$).
      First note that $S(t) \to \infty$ as $t \to \infty$, because extinction is sure.
      In the case $T<s < S(t)$, the mean value theorem tells us
        \begin{align*}
          g_{s}''(1-\phi(s,t)) & = g_{s}''(1)-g_{s}'''(\xi)\phi(s,t), \qquad \xi \in (1,1-\phi(s,t)) \\
          & \geq g_{s}''(1)-L\phi(s,t) \\
          & = g_{s}''(1)(1-L\phi(s,t)/g_{s}''(1)) \\
          & \geq g_{s}''(1)(1-L\phi(s,t)/\eta) \\
          & \geq (1-\epsilon)g_{s}''(1).
        \end{align*}
    \item At large $t$, (4) and our result in 2. tell us
        \begin{align*}
          \frac{1}{\phi(t)\Gamma(t)} & \geq \frac{1}{\Gamma(t)}\int_{0}^{t}
            \frac{g_{s}''(1-\phi(s,t))}{2M(s)}ds \\
          & \geq \frac{1}{\Gamma(t)}\int_{T}^{S(t)}
            \frac{g_{s}''(1-\phi(s,t))}{2M(s)}ds \\
          & \geq \frac{1}{\Gamma(t)}(1-\epsilon)\int_{T}^{S(t)}\frac{g_{s}''(1)}{2M(s)}ds
        \end{align*}
      Hypothesis (H10) lets us ignore the loss of the $[0,T]$ interval, so Lemma 4 tells us that 
        \begin{align*}
          \frac{1}{\Gamma(t)}(1-\epsilon)\int_{T}^{S(t)}\frac{g_{s}''(1)}{2M(s)}ds
            & = \frac{1}{\Gamma(t)}(1-\epsilon)\Gamma(S(t)) + o\left(\frac{\Gamma(S(t))}{\Gamma(t)}\right) \\
            & = \frac{\Gamma(S(t))}{\Gamma(t)}-\epsilon\frac{\Gamma(S(t))}{\Gamma(t)}
              + o(1)\\
            & \geq \frac{\Gamma(S(t))}{\Gamma(t)}-\epsilon +o(1)\\
            & = (1-\epsilon)-\frac{\Gamma(t)-\Gamma(S(t))}{\Gamma(t)}+o(1).
        \end{align*}
      We want to show that the $\Gamma$ term tends to zero.
    \item Our result in 1. and the definition of $S(t)$ tell us
        \begin{align*}
          \frac{\Gamma(t)-\Gamma(S(t))}{\Gamma(t)} & \leq \frac{C}{M(S(t))\phi(S(t),t)\Gamma(t)} \\
          & \leq \frac{CL}{\epsilon \eta} \frac{1}{M(S(t))\Gamma(t)} \\
          & \leq \frac{CL}{\epsilon \eta}\frac{1}{M(S(t))\Gamma(S(t))},
        \end{align*}
      which tends to zero because $S(t) \to \infty$. Combining this with 3., we find
        \begin{align*}
          \liminf_{t\geq 0} \frac{1}{\phi(t)\Gamma(t)} \geq 1-\epsilon,
        \end{align*}
      which, along with Theorem 7 part 3., tells us
        \begin{align*}
          1 \leq \liminf_{t \geq 0} \phi(t)\Gamma(t) \leq \limsup_{t \geq 0} \phi(t)\Gamma(t) \leq \frac{1}{1-\epsilon},
        \end{align*}
      for every $\epsilon>0$, completing the proof.
  \end{enumerate}
  \end{proof}

\subsection{Proof of Theorem 9} 
\label{Proof of Theorem 9}
In Theorem 9 we depart more significantly from the discrete setting. We still prove the claim
by induction, but now we rely on the differential formula of Lemma 3.

  \begin{proof}
  \begin{enumerate}
    \item We proceed by induction. The base case $r=1$ is clear. Now assume that
        \begin{align*}
          \frac{M_{j}(t)}{M^{j}(t)} = j!\Gamma(t)^{j-1}+o(\Gamma(t)^{j-1}) \qquad (j=1,...,r-1).
        \end{align*}
      Recall from Lemma 1 that $\dot{M} = ME(X_{t})$. Using Lemma 1, we find
        \begin{align*}
          (M_{r}/M^{r})'(t) & = \frac{1}{M^{r}}\sum_{j=1}^{r}\binom{r}{j}M_{r-j+1}E(X_{t}^{j})
            -r\frac{M_{r}}{M^{r+1}}\dot{M} \\
          & = \frac{1}{M^{r}}\sum_{j=1}^{r}\binom{r}{j}M_{r-j+1}E(X_{t}^{j})
            -\frac{1}{M^{r}}rM_{r}E(X_{t}) \\
          & = \frac{1}{M^{r}}\sum_{j=2}^{r}\binom{r}{j}M_{r-j+1}E(X_{t}^{j}).
        \end{align*}
      Then our inductive hypothesis tells us
        \begin{align*}
          (M_{r}/M^{r})'(t) & = \sum_{j=2}^{r}\binom{r}{j}E(X_{t}^{j})(r-j+1)! \frac{\Gamma(t)^{r-j}}{M(t)^{j-1}}
            + o\left(\frac{\Gamma(t)^{r-j}}{M(t)^{j-1}}\right).
        \end{align*}
      Hypothesis (H11) tells us that the term according to $j=2$ dominates the sum
      at large times, and thus
        \begin{align*}
          (M_{r}/M^{r})'(t) & = \frac{r(r-1)}{2} E(X_{t}^{2})\frac{(r-1)!\Gamma(t)^{r-2}}{M(t)}
            + o\left(\frac{\Gamma(t)^{r-2}}{M(t)}\right) \\
          & = r!(r-1)\Gamma(t)^{r-2}\frac{E(X_{t}^{2})}{2M(t)}.
        \end{align*}
  \item On the other hand, differentiating $\Gamma(t)^{r-1}$ yields
      \begin{align*}
        (\Gamma^{r-1})'(t) & = (r-1)\Gamma(t)^{r-2}\dot{\Gamma}(t) \\
        & = (r-1)\Gamma(t)^{r-2}\frac{E(X_{t}^{2})}{2M(t)}.
      \end{align*}
    Combining this with 1., we find
      \begin{align*}
        (M_{r}/M^{r})'(t) = r!(\Gamma^{r-1})'(t) + o((\Gamma^{r-1})'(t)),
      \end{align*}
    and integrating both sides yields the desired result.
  \end{enumerate}
  \end{proof}

\subsection{Proof of Theorem 6} 
\label{Proof of Theorem 6}
As in the discrete setting, proving the main theorem is now just a matter
of putting together the pieces of the supporting theorems.

  \begin{proof}
    The moments of the conditioned, scaled process can be written as
      \begin{align*}
        E\left(\left(\frac{Z_{t}}{E(Z_{t}|Z_{t} \ne 0)}\right)^{r}\middle|Z_{t} \ne 0\right)
          & = \frac{E(Z_{t}^{r}|Z_{t} \ne 0)}{E(Z_{t}|Z_{t} \ne 0)^{r}} \\
        & = \frac{M_{r}(t)/\phi(t)}{(M(t)/\phi(t))^{r}} \\
        & = \frac{M_{r}(t)\phi(t)^{r-1}}{M(t)^{r}}.
      \end{align*}
    Theorems 8 and 9 tell us
      \begin{align*}
        \lim_{t \to \infty} \frac{M_{r}(t)\phi(t)^{r-1}}{M(t)^{r}}
          = \lim_{t \to \infty} \frac{M_{r}(t)}{M(t)^{r}\Gamma(t)^{r-1}} = r!.
      \end{align*}
    Thus the $r$-th moment of the conditioned, scaled process tends to $r!$ and so
    the process converges to an exponential with parameter 1 in moments.
  \end{proof}

\subsection{Proof of Theorem 10} 
\label{Proof of Theorem 10}

  \begin{proof}
    First we calculate
      \begin{align*}
        \int_{0}^{t} \frac{E(X_{s})}{M(s)}ds = -\int_{0}^{t}\frac{d}{ds} \frac{1}{M(s)}ds
          = 1-\frac{1}{M(t)}.
      \end{align*}
    Thus, $M(t) = \left(1-\int_{0}^{t}\frac{E(X_{s})}{M(s)}ds\right)^{-1}$. Let $\epsilon>0$
    and pick $T$ large enough that $|E(X_{t})| < \epsilon$ for all $t \geq T$. Then using
    hypotheses (H8) and (H9) to bound $E(X_{t}^{2})$ from below and above, when $t \geq T$ we calculate
      \begin{align*}
        \frac{1}{\Gamma(t)M(t)} & \leq \frac{1}{\Gamma(t)}
          + C\frac{\int_{0}^{t}\frac{|E(X_{s})|}{M(s)}ds}{\int_{0}^{t}\frac{1}{M(s)}ds} \\
        & \leq \frac{1}{\Gamma(t)} + \frac{C\int_{0}^{T}\frac{E(X_{s})}{M(s)}ds}{\int_{0}^{t}\frac{1}{M(s)}ds}
          + \frac{C\int_{T}^{t}\frac{E(X_{s})}{M(s)}ds}{\int_{T}^{t}\frac{1}{M(s)}ds} \\
        & \leq \frac{1}{\Gamma(t)} + \frac{C}{\Gamma(t)}
          + C\epsilon.
      \end{align*}
    The first two terms tend to zero by hypothesis (H10), and the result follows.
  \end{proof}

\end{document}